\documentclass[12p]{amsart}
\usepackage[margin=1in]{geometry}

\usepackage[utf8]{inputenc}
\usepackage[mathscr]{eucal}
\usepackage{enumerate}
\usepackage{amssymb}
\usepackage{amsmath}
\usepackage{amsthm}
\usepackage{bm}
\usepackage{url}
\usepackage{arydshln}
\usepackage{pgfplots}
\usepackage{hyperref}

\usepackage{pgf,tikz}
\usetikzlibrary{arrows,matrix}
\usetikzlibrary{arrows.meta,math}

\numberwithin{equation}{section}
\theoremstyle{plain}
\newtheorem{thm}{Theorem}[section]
\newtheorem{lemma}[thm]{Lemma}
\newtheorem{prop}[thm]{Proposition}
\newtheorem{alg}[thm]{Algorithm}

\newtheorem{open}[thm]{Question}

\theoremstyle{definition}
\newtheorem{defn}[thm]{Definition}

\newtheorem{remark}[thm]{Remark}
\newtheorem{example}[thm]{Example}

\newcommand{\Z}{\mathbb {Z}}

\newcommand{\mbf}{\mathbf}

\DeclareMathOperator{\im}{im}
\DeclareMathOperator{\coker}{coker}

\author{Michael Ion} 
\address{School of Information, University of Michigan \\
Ann Arbor, Michigan 48104, USA}
\email{mikeion@umich.edu}

\title{Chip-Firing and the Sandpile Group of the $R_{10}$ Matroid}
\author{Alex McDonough}
\address{Department of Mathematics, University of Oregon\\
Eugene, Oregon 97403, USA}
\email{alexmcd@uoregon.edu}

\date{\today}

\begin{document}

\begin{abstract}A celebrated result of Seymour is that all \emph{regular matroids} are built up from \emph{graphic matroids}, \emph{cographic matroids}, and a specific 10 element rank 5 matroid called $R_{10}$. In this article, we give a simple description of \emph{chip-firing} on $R_{10}$ using complex numbers on the vertices of a pentagon, and link to an app where readers can play around with the combinatorial dynamics of the system. We also provide an easy to describe set of representatives for each of the 162 equivalence classes that make up the \emph{sandpile group} of $R_{10}$. \end{abstract}

\maketitle

\section{Introduction}

Several foundational results in \emph{chip-firing}, most notably the \emph{matrix-tree theorem}, have been shown to generalize naturally from graphs to \emph{regular matroids}~\cite{Merino}. This is possible through the perspective of \emph{circuit} and \emph{cocircuit} lattices. However, we do not know of any examples in the literature which focus on chip-firing properties of \emph{regular matroids} as more than a special case of \emph{graphic matroids}. 

In 1980, Seymour showed that every regular matroid can be formed from \emph{1-sums}, \emph{2-sums}, and \emph{3-sums} of graphic matroids, co-graphic matroids, and a specific matroid called $R_{10}$~\cite{Seymour}. For this reason, we suspect that understanding chip-firing on $R_{10}$ could be the key to understanding nuances of chip-firing on regular matroids in general. 

The term \emph{chip-firing} can mean different things to different authors, but it generally refers to a discrete dynamical system that can be described in terms of a \emph{configuration} of objects (often visualized as poker chips) which can be rearranged through operations called \emph{firing moves}. See~\cite{Klivans} for many more details about the history and applications of chip-firing.

For this article, we use chip-firing to understand equivalence classes of the \emph{sandpile group} of $R_{10}$. The \emph{sandpile group} of a graph, (which also goes by many other names, such as the \emph{critical group}, or \emph{Jacobian}) is a finite abelian group whose cardinality is equal to the number of spanning trees~\cite{Biggs99}. The sandpile group of a \emph{regular matroid} is a finite abelian group whose cardinality is equal to the number of \emph{bases}~\cite{Merino}. To compute the sandpile group of a regular matroid, we take the quotient of the free abelian group on the ground set by an equivalence relation given by a specific collection of firing moves. 

Using linear algebra, one can show that the sandpile group of $R_{10}$ is equal to $(\Z/3\Z)^4 \oplus (\Z/2\Z)$, a finite abelian group of order 162. Furthermore, the matroid $R_{10}$ is equipped with certain symmetries that allow for an elegant chip-firing description using both \emph{real} and \emph{imaginary} chips on a pentagon. More formally, we can work over the ring of \emph{Gaussian integers} $\Z[i]$ instead of the ring of integers $\Z$. Using this description, we give a set of representatives for all of the equivalence classes of $S(R_{10})$, as well as an algorithm from an arbitrary configuration to one of these representatives. Specifically, we show that every firing equivalence class has a unique representatives with 0 or 3 real chips on a distinguished node, 0, 1, or 2 real chips on each of four other nodes, and no imaginary chips. 

We also designed a web app to allow people to better understand the dynamics of the system through interactive exploration. The app can be found here: \url{https://pentagon-game.github.io/}, and a screenshot is shown in Figure~\ref{fig:app_example}. Both authors found the app to be incredibly helpful as a tool to build intuition, and we hope that readers will feel the same way. 

\section{Background} 

For this article, we do not need the full generality of matroid theory. Instead, we can restrict to \emph{regular matroids}, which can be described in terms of a matrix.

Let $A$ be a matrix, $n$ be the number of columns of $A$, and $F$ be a field. The \emph{ground set} $E$ is a set whose elements correspond to the columns of $A$. If a subset of $E$ corresponds to columns which are linearly independent over $F$, then it is called an \emph{independent} set. The maximal independent sets are called \emph{bases}, and the collection of bases is called $\mathbf B$.

A matroid is \emph{representable} if it can be represented by some matrix over some field. The classical definition of a \emph{regular matroid} is a matroid that is representable over \emph{any} field. We will use the following definition, which was proven equivalent by Tutte~\cite{Tutte}. 
\begin{defn}
    A \emph{regular matroid} is a matroid that is representable over $\mathbb R$ by a \emph{totally unimodular} matrix $A$, where a \emph{totally unimodular} matrix is an integer matrix such that every minor is in the set $\{0, 1,-1\}$. 
\end{defn}

It is well known that any graph $G$ corresponds to a regular matroid $M(G)$, whose ground set corresponds to the edges of $G$, and whose bases correspond to \emph{spanning trees} of $G$~\cite[Section 5.1]{Oxley}. However, the other direction does not hold: there exist regular matroids whose bases do not correspond to the spanning trees of a graph. One such matroid is $R_{10}$, which will be the main focus of Section~\ref{sec:R10}. 

Let $M = (E,\mbf B)$ be a regular matroid represented by a totally unimodular matrix $A$ over $\mathbb R$. Suppose that we assign an integer number of \emph{chips} to each element of the ground set of $E$.  We call such an assignment a \emph{chip configuration}, and represent it as an element of the lattice $\Z^E$.

Let $\im_{\Z} A^t$ be the sublattice of $\Z^E$ made up of integer linear combinations of rows of $A$ (which are also the columns of $A^t$). Let $\ker_{\Z} A$ be the sublattice of $\Z^E$ made up of vectors mapped to the all zeros vector by $A$. 

\begin{defn}\label{def:sandpile}
    Suppose that $M$ is a matroid represented by a totally unimodular matrix $A$ over $\mathbb R$. The \emph{sandpile group} of $M$ is given by:
    \[S(M) := \frac{\Z^E}{\im_ {\Z} A^t\oplus \ker_ {\Z} A}\]
\end{defn}

Note that Definition~\ref{def:sandpile} does not meaningfully depend on the choice of totally unimodular matrix $A$ representing $M$, beyond reordering or reversing the orientation of the elements of the ground set. For further discussion, see~\cite[Section 2.1]{BBY} and~\cite{SuWagner}. 

\begin{remark}\label{rem:names}
    There are a lot of different names for the group given in Definition~\ref{def:sandpile}, and these names are not always used interchangeably. For example, the main group considered in~\cite{BBY} is called the \emph{Jacobian group}, which is defined as the \emph{determinant group} of $\im_ {\Z} A^t$. Nevertheless, this group is shown to be isomorphic to the group in Definition~\ref{def:sandpile}. In~\cite{CutsFlows}, the authors generalize to work with \emph{representable matroids} which may not be regular. In this context, Definition~\ref{def:sandpile} is called the \emph{cutflow group}, while the determinant group of $\im_{\Z} A^t$ is called the \emph{critical group}. In our context, these two groups are isomorphic, so we do not make a distinction in our naming. The name \emph{sandpile group} comes from the Bak--Tang--Wiesenfeld \emph{sandpile model} of \emph{self-organized criticality}~\cite{BTW}.
\end{remark}

\begin{remark}\label{rem:mtt}
    To keep this article concise, we will focus on \emph{how} to describe the sandpile group of a regular matroid (in particular, of $R_{10}$), and will say very little about \emph{why} this group is interesting. Nevertheless, we do feel it is worth mentioning one important property of this group. For any regular matroid $M$, the size of the sandpile group $S(M)$ is equal to the number of bases of $M$~\cite[Theorem 3.3.2]{Merino}. This result is a variation of \emph{Kirchhoff's matrix-tree theorem} and can be proven by applying the \emph{Cauchy-Binet formula}. 
\end{remark}

It can be a bit cumbersome to work with both the image and the kernel of a matrix, but there is a trick to rewrite $\im_ {\Z} A^t\oplus \ker_ {\Z} A$ as the image of a larger matrix. Throughout this article, we use $n$ for the size of the ground set of a regular matroid $M$, and $r$ for the \emph{rank} of $M$, which is equal to the minimal number of rows of a matroid representing $M$.

Suppose that an $r \times n$ totally unimodular matrix $A$ is of the form \begin{equation}\label{eq:stand_form} A = \begin{bmatrix} I_r & D\end{bmatrix},\end{equation}
where $I_r$ is the $r\times r$ identity matrix and $D$ is an $r \times (n-r)$ integer matrix. Note that $A$ can always be put in this form by choosing a basis for the first $r$ columns, and then applying row operations. Then, we define \begin{equation}\label{eq:dual_form}\widehat A = \begin{bmatrix} D^t & -I_{n-r}\end{bmatrix},\end{equation} where $D^t$ is the transpose of $D$ and $I_{n-r}$ is again the identity matrix. Additionally, we define \begin{equation}\label{eq:combo_form}K = \begin{bmatrix} A \\ \widehat A \end{bmatrix} = \begin{bmatrix}I_r & D \\D^t & -I_{n-r}\end{bmatrix},\end{equation}

It is well-known that the matroid represented by $\widehat A$ is \emph{dual} to the matroid represented by $A$~\cite[Theorem 2.2.8]{Oxley}. This implies that $\ker_{\Z}  A = \im_{\Z} \widehat A^t$. It follows that if $A$ represents a regular matroid $M$, then 
\begin{equation}\label{eq:sandpile_def}S(M) = \frac{\Z^E}{\im_ {\Z} A^t\oplus \ker_ {\Z} A} = \frac{\Z^E}{\im_ {\Z} A^t\oplus \im_ {\Z} \widehat A^t} = \frac{\Z^E}{\im_{\Z}K^t} = \frac{\Z^E}{\im_{\Z}K}\hspace{.2 cm}\text{, where } K = \begin{bmatrix} I_r & D \\ D^t & -I_{n-r}\end{bmatrix}.\end{equation}

This quotient group $\Z^E/\im_{\Z}K$ is called the \emph{cokernel} of $K$ over $\mathbb Z$, and is written $\coker_\Z K$. The following proposition gives a summary of our observations after Remark~\ref{rem:names}. 

\begin{prop}~\cite[Definition 3.11]{Multijection}
     Let $M$ be a regular matroid with ground set $E$ that is represented over $\mathbb R$ by a totally unimodular matrix $A = \begin{bmatrix} I_r & D\end{bmatrix}$. Furthermore, define $\widehat A= \begin{bmatrix} D^t & -I_{n-r}\end{bmatrix}$ and $K = \begin{bmatrix} A \\ \widehat A\end{bmatrix} = \begin{bmatrix} I_r & D \\ D^t & -I_{n-r}\end{bmatrix}$. The \emph{sandpile group} $S(M)$ is equal to $\coker_{\mathbb Z} K$ (which is defined as $\Z^E/\im_{\Z} K$). 
\end{prop}

\begin{remark}
    It is common to use $-\widehat A$ in place of $\widehat A$. In fact, this is the version used in both~\cite[Theorem 2.2.8]{Oxley} and \cite[Definition 3.11]{Multijection}. However, the setup from~\eqref{eq:combo_form} has a significant advantage for our purposes: the matrix $K$ is \emph{symmetric}. This property prevents potential confusion between the role of rows and columns of $K$, and allows us to write $K$ in places where we would otherwise write $K^t$. Note that $\im_{\mathbb Z} K^t$ does not change when $\widehat A$ is replaced with $-\widehat A$. 
\end{remark}

One way to think of the sandpile group of $M$ is in terms of \emph{chip configurations} and \emph{firings moves}. Begin by thinking of an element $C \in \Z^E$ as an assignment of an integer number of \emph{chips} assigned to each element of $E$. If we add or subtract a row of $K$, this redistributes the chips to give a new chip configuration $C'$ such that $C$ and $C'$ are equivalent as elements of $S(M)$. We say that adding or subtracting a row of $K$ is a \emph{firing move} on the chip configuration, and we say that two chip configurations are \emph{firing equivalent} if one can be reached from the other by a sequence of firing moves. Algebraically speaking, chip configurations $C$ and $C'$ are firing equivalent if and only if $C - C' \in \im_{\Z} K$. Equivalently, this holds precisely when $K^{-1}(C-C')$ is an integer vector. 

\begin{remark}\label{rem:traditional}
    Typically, \emph{chip configurations} refer to a collection of chips on the vertices of a graph, instead of on the ground set of a regular matroid. In this context, a \emph{firing move} removes $\deg(v)$ chips from a vertex $v$, and sends one to each neighbor of $v$. It is also common to disallow firings which result in a negative number of chips on the firing vertex. For more information about this more traditional perspective, see~\cite{Klivans}. 
    
    In our construction, firings are more abstract, and the total number of chips is not necessarily preserved. Furthermore, we place no restrictions on when a firing is allowed, and it is not a problem for elements of $E$ to have negative chips. 
\end{remark}

Using linear algebra techniques, we can calculate $S(M)$ precisely. 

\begin{prop}\cite[Section 8]{Biggs99}\label{prop:smith_normal}
    For any square integer matrix $K$, the group $\coker_{\Z} K$ has cardinality equal to $|\det(K)|$. More precisely, let $a_1,a_2, \dots, a_n$ be the diagonal entries of the \emph{smith normal form} of $K$. Then, $\coker_{\Z}K \cong (\Z/{a_1}\Z) \oplus (\Z/{a_2}\Z) \oplus \cdots \oplus (\Z/{a_n}\Z)$.
\end{prop} 

In the next section, we will focus on specific properties of $R_{10}$, but first let us give a simple example of chip-firing on a regular matroid. 

\begin{example}
    Let $M$ be the regular matroid represented by the matrix 
    \[A = \begin{bmatrix} 1 & 0 & -1 \\ 0 & 1 & -1\end{bmatrix}.\hspace{ .3 cm} \text{ This means that } \widehat A = \begin{bmatrix} -1 & -1 & -1\end{bmatrix}\hspace{ .3 cm} \text{ and } \hspace{ .3 cm} K = \begin{bmatrix} 1 & 0 & -1 \\ 0 & 1 & -1 \\ -1 & -1 & -1\end{bmatrix}.\]
    By computing the Smith normal form of $K$, one can calculate that $S(M) = \Z/3\Z$, which means that there are 3 firing equivalence classes of chip configurations. 
    
    A chip configuration on $M$ can be written as a triple of integers $(a,b,c)$. We claim that each chip configuration is firing equivalent to exactly one of the configurations $(0,0,0)$, $(1,0,0)$, or $(2,0,0)$. In other words, these form a set of representatives for the elements of $S(M)$. One way to prove this claim is through linear algebra, noting that chip configurations $C$ and $C'$ are firing equivalent if and only if $K^{-1}(C-C') \in \Z^3$. 

    Alternatively, we can also prove the claim by showing how to go from an arbitrary configuration to one of the three representatives using a sequence of firings. In particular, suppose that we are given an arbitrary chip configuration $(a,b,c)$. First, we can apply the $(-1,-1,-1)$ firing move a total of $\lfloor(a+b+c)/3\rfloor$ times. Since each firing move subtracts 3 chips, the resulting configuration will have a total of $0$, $1$, or $2$ chips. Next, we can add or subtract $(0,1,-1)$ until no chips remain in the second position. Finally, we can add or subtract $(1,0,-1)$ until no chips remain in the third position. These firing moves do not change the total number of chips, so we end up with $0$, $1$, or $2$ chips in the first position, and no chips on the other two positions. 
\end{example}

\section{Chip-firing on $R_{10}$}\label{sec:R10} 

Consider the totally unimodular matrix
\begin{equation}\label{eq:R10}\mathcal A = \begin{bmatrix}
    1 & 0 & 0 & 0 & 0 &   1 & -1 & 0 & 0 & -1\\
    0 & 1 & 0 & 0 & 0 &   -1 & 1 &-1 & 0 & 0\\
    0 & 0 & 1 & 0 & 0 &   0 & -1 & 1 & -1 & 0\\
    0 & 0 & 0 & 1 & 0 &   0 & 0 & -1 & 1 & -1\\
    0 & 0 & 0 & 0 & 1 &   -1 & 0 & 0 & -1 & 1\\
\end{bmatrix}.\end{equation}

\begin{defn}\label{def:R10}
    The matroid represented by $\mathcal A$ over $\mathbb R$ is called $R_{10}$. 
\end{defn}

We can write $\mathcal A$ as the block matrix $\begin{bmatrix} I_5 & \mathcal D \end{bmatrix}$, where $\mathcal D$ is the $5 \times 5$ matrix defined by:

\[\mathcal D_{ij} = \begin{cases} 1 & \text{if $i=j$,}\\
-1 &\text{if $i-j \equiv 1\pmod 5$,}\\
0 & \text{otherwise.}\end{cases}\]

Since $\mathcal A$ is in the format from~\eqref{eq:stand_form}, we can use~\eqref{eq:dual_form} to write $\widehat{\mathcal A} = \begin{bmatrix}\mathcal D^t & -I_5 \end{bmatrix} = \begin{bmatrix} \mathcal D & -I_5 \end{bmatrix}$, where the final equality holds because $\mathcal D$ is symmetric. Let $\mathcal K$ be the matrix 

\[\mathcal K = \begin{bmatrix} \mathcal A\\ \widehat{\mathcal A} \end{bmatrix} = \begin{bmatrix} I_5 & \mathcal D \\ \mathcal D & -I_5 \end{bmatrix}= \begin{bmatrix}
    1 & 0 & 0 & 0 & 0 &   1 & -1 & 0 & 0 & -1\\
    0 & 1 & 0 & 0 & 0 &   -1 & 1 & -1 & 0 & 0\\
    0 & 0 & 1 & 0 & 0 &   0 & -1 & 1 & -1 & 0\\
    0 & 0 & 0 & 1 & 0 &   0 & 0 & -1 & 1 & -1\\
    0 & 0 & 0 & 0 & 1 &   -1 & 0 & 0 & -1 & 1\\
    1 & -1 & 0 & 0 & -1 &  -1 & 0 & 0 & 0 & 0\\
    -1 & 1 & -1 & 0 & 0 &  0 & -1 & 0 & 0 & 0\\
    0 & -1 & 1 & -1 & 0 &  0 & 0 & -1 & 0 & 0\\
    0 & 0 & -1 & 1 & -1 &  0 & 0 & 0 & -1 & 0\\
    -1 & 0 & 0 & -1 & 1 &  0 & 0 & 0 & 0 & -1\\
\end{bmatrix}\]

By Proposition~\ref{prop:smith_normal}, and the help of a computer algebra system, we can calculate that $|S(R_{10})| = 162$. More precisely, after computing the Smith normal form, we find that $S(R_{10}) \cong (\Z/3\Z)^4 \oplus (\Z/2\Z)$.

While these calculations can be done for any regular matroid, the fact that $\mathcal D$ is not only square, but also symmetric, allows for an alternate perspective using the \emph{Gaussian integers} $\Z[i]$.

Let $\overline{\mathcal K}$ be the $5\times 5$ matrix in $\Z[i]^5$ given by 
\[\overline{\mathcal K} =  I_5 + \mathcal D i =\begin{bmatrix}
    1 + i & -i & 0 & 0 & -i \\
    -i & 1 + i & -i & 0 & 0 \\
    0 & -i & 1 + i & -i & 0 \\
    0 & 0 & -i & 1 + i & -i \\
    -i & 0 & 0 & -i & 1 + i \\
\end{bmatrix}\]

\begin{prop}\label{prop:samecoker}
    The sandpile group $S(R_{10})$ is isomorphic to $\coker_{\Z[i]} (\overline{\mathcal K})$. 
\end{prop}
\begin{proof}
    We need to show that $\coker_{\Z[i]} (\overline{\mathcal K})$ and $\coker_{\Z} ({\mathcal K})$ are isomorphic. To prove this, we will show that for any pair $\mbf v,\mbf w \in \Z^5$, the combined vector $(\mbf v,\mbf w)$ is in $\im_{\Z} \mathcal K$ if and only if $\mbf v + \mbf w i$ is in $\im_{\Z[i]} \overline{\mathcal K}$. 

    Consider a pair $\mbf a,\mbf b \in \Z^5$. One can calculate directly that
    \[\mathcal K \begin{bmatrix}\mbf a \\ \mbf b \end{bmatrix} = \begin{bmatrix}\mbf a + \mathcal D\mbf b \\  -\mbf b + \mathcal D\mbf a\end{bmatrix} \hspace{ 1 cm}\text{ and }\]
    \[\overline{\mathcal K} (\mbf a - \mbf b i) = (I_5 + \mathcal D i) (\mbf a - \mbf b i) = (\mbf a + \mathcal D\mbf b) +  (-\mbf b + \mathcal D\mbf a)i.\]
    The result follows. 
\end{proof}

Proposition~\ref{prop:samecoker} introduces an approachable way to interpret $S(R_{10})$. Imagine a pentagon such that each node can have two kinds of chips, real and imaginary.\footnote{We use the term \emph{node} in place of \emph{vertex} to help distinguish our version of chip-firing from the more traditional version on a graph (which is mentioned in Remark~\ref{rem:traditional}). Each node of the pentagon corresponds to a pair of elements of the ground set of $R_{10}$. Note that the ground set of a matroid is an abstraction of the \emph{edges} of a graph, not the \emph{vertices}.} There are also four kinds of firing moves, which only differ by multiplication by $i$. Each of these moves impacts the chips on three adjacent nodes.

\begin{itemize}
    \item[(A)~] Add $1+i$ to a node and add $-i$ to each neighboring node. 
    \item[(B)~] Add $-1+i$ to a node and add $1$ to each neighboring node. 
    \item[(-A)] Add $-1-i$ to a node and add $i$ to each neighboring node. 
    \item[(-B)] Add $1-i$ to a node and add $-1$ to each neighboring node. 
\end{itemize}

We will call these moves \emph{A firings}, \emph{B firings}, \emph{(-A) firings}, and \emph{(-B) firings} respectively. Two chip-configurations on the pentagon are firing equivalent if and only if one can be reached from the other by some sequence of A, B, (-A), and (-B) firings. In Figure~\ref{fig:app_example}, we show a screenshot from our app, which displays the chip configuration $(2-2i, 1, -2-i, 3+i, -1 + i)$. The user can redistribute the chips using the various firings. 

\begin{figure}
    \centering
    \includegraphics[width=.5\linewidth]{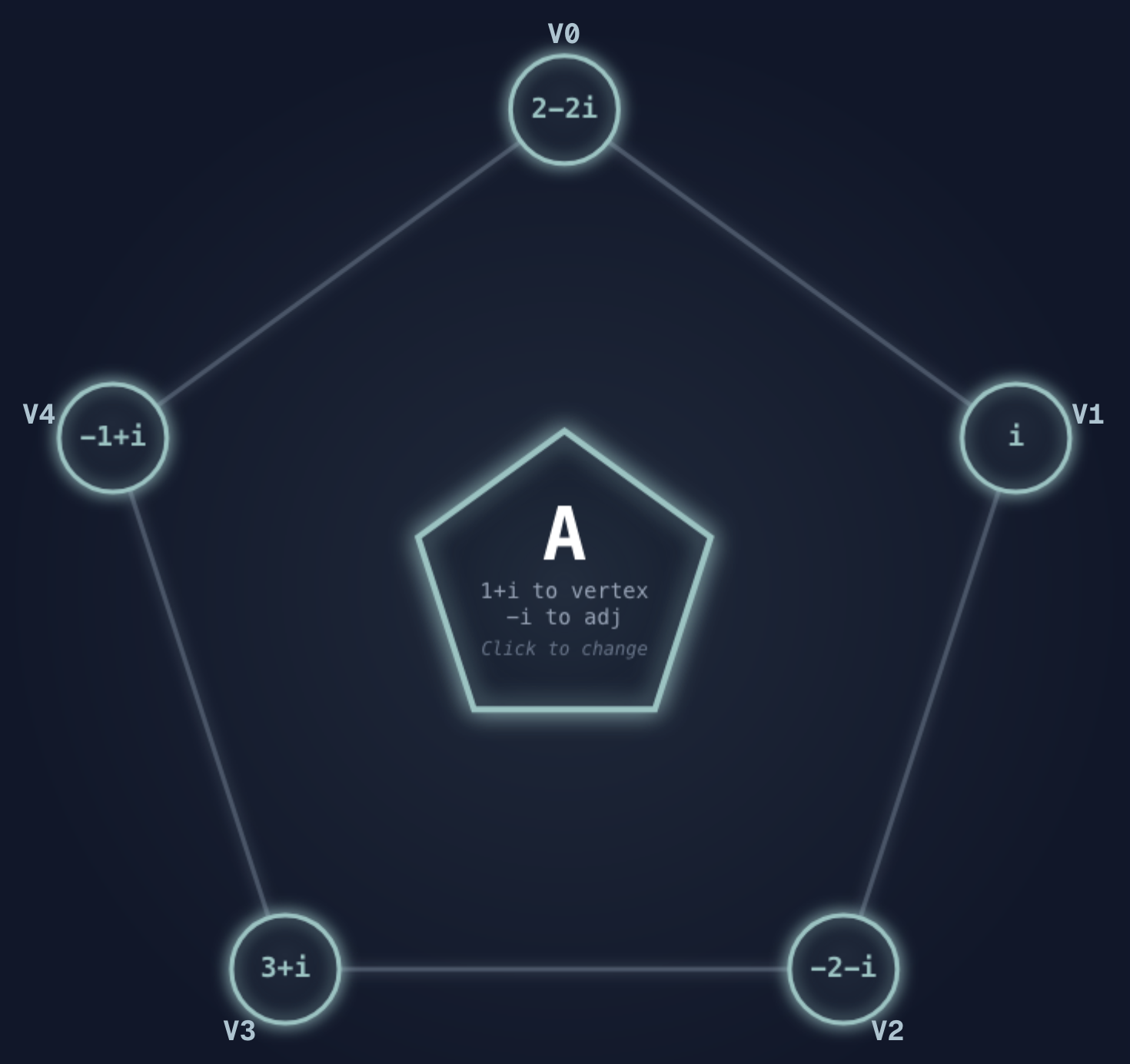}
    \caption{Above is a screenshot from the app that we developed to help understand chip-firing on $R_{10}$ \url{https://pentagon-game.github.io/}. The user can fire a vertex by clicking on it, or switch the kind of firing by clicking on the center pentagon. In this game mode, the chip configuration is chosen to be firing equivalent to the all zeros configuration, and the goal is to reach this configuration through a sequence of firings. }
    \label{fig:app_example}
\end{figure}

In addition to the perspective given by the app, we feel it may be helpful to briefly describe an alternate perspective. Instead of a collection of real and imaginary chips on a pentagon, chip configurations on $R_{10}$ can also be visualized as collections of nodes labeled 0 through 4 on lattice points on the complex plane (see Figure ~\ref{fig:geometric}). A firing move consists of moving one node diagonally $\sqrt2$ units to another lattice point, and moving the two nodes with adjacent labels (mod 5) one unit in a direction $3\pi/4$ radians clockwise from the chosen node. 

\begin{figure}
\begin{center}
    \begin{tikzpicture}[scale = .7, anno/.style = {very thick,-{Stealth[length = 2.5mm, width = 2.2mm, open]}}]
        \draw [thick,<->,>=stealth'] (-3.5,0) to (3.5,0);
        \draw [thick,<->,>=stealth'] (0,-3.5) to (0,3.5);
        \foreach \x in {-3,...,3}
            {\draw [dotted] (\x,-3.5) to (\x,3.5);
            \draw [dotted] (-3.5,\x) to (3.5,\x);}
    \tikzstyle{o}=[circle,draw,fill=black, text = white,inner sep=1pt,minimum size = 1.8mm]
    \tikzstyle{k}=[circle,draw,fill = white,minimum size = 1.8mm]
    \node[o] (0) at (1,1) {\Large{$0$}};
    \node[o] (1) at (-2,3) {\Large{$1$}};
    \node[o] (2) at (-2,0) {\Large{$2$}};
    \node[o] (3) at (1,-2) {\Large{$3$}};
    \node[o] (4) at (2,0) {\Large{$4$}};

    \begin{scope}[shift = {(8,0)}]
    \draw [thick,<->,>=stealth'] (-3.5,0) to (3.5,0);
        \draw [thick,<->,>=stealth'] (0,-3.5) to (0,3.5);
        \foreach \x in {-3,...,3}
            {\draw [dotted] (\x,-3.5) to (\x,3.5);
            \draw [dotted] (-3.5,\x) to (3.5,\x);}
    \node[o] (0) at (2,2) {\Large{$0$}};
    \node[o] (1) at (-2,2) {\Large{$1$}};
    \node[o] (2) at (-2,0) {\Large{$2$}};
    \node[o] (3) at (1,-2) {\Large{$3$}};
    \node[o] (4) at (2,-1) {\Large{$4$}};

    \node[k] (0') at (1,1){};
    \node[k] (1') at (-2,3){};
    \node[k] (4') at (2,0){};

    \draw [anno] (0') to (0);
    \draw [anno] (1') to (1);
    \draw [anno] (4') to (4);
    \end{scope}
    
     \begin{scope}[shift = {(16,0)}]
    \draw [thick,<->,>=stealth'] (-3.5,0) to (3.5,0);
        \draw [thick,<->,>=stealth'] (0,-3.5) to (0,3.5);
        \foreach \x in {-3,...,3}
            {\draw [dotted] (\x,-3.5) to (\x,3.5);
            \draw [dotted] (-3.5,\x) to (3.5,\x);}
    \tikzstyle{o}=[circle,draw,fill=black, text = white,inner sep=1pt,minimum size = 1.8mm]
    \node[o] (0) at (2,2) {\Large{$0$}};
    \node[o] (1) at (-3,2) {\Large{$1$}};
    \node[o] (2) at (-1,-1) {\Large{$2$}};
    \node[o] (3) at (0,-2) {\Large{$3$}};
    \node[o] (4) at (2,-1) {\Large{$4$}};

    \node[k] (2') at (-2,0){};
    \node[k] (1') at (-2,2){};
    \node[k] (3') at (1,-2){};

    \draw [anno] (2') to (2);
    \draw [anno] (1') to (1);
    \draw [anno] (3') to (3);
    \end{scope}

    \begin{scope}[shift = {(0,-8)}]
    \draw [thick,<->,>=stealth'] (-3.5,0) to (3.5,0);
        \draw [thick,<->,>=stealth'] (0,-3.5) to (0,3.5);
        \foreach \x in {-3,...,3}
            {\draw [dotted] (\x,-3.5) to (\x,3.5);
            \draw [dotted] (-3.5,\x) to (3.5,\x);}
    \tikzstyle{o}=[circle,draw,fill=black, text = white,inner sep=1pt,minimum size = 1.8mm]
    \node[o] (0) at (2,3) {\Large{$0$}};
    \node[o] (1) at (-3,2) {\Large{$1$}};
    \node[o] (2) at (-1,-1) {\Large{$2$}};
    \node[o] (3) at (0,-1) {\Large{$3$}};
    \node[o] (4) at (1,-2) {\Large{$4$}};

    \node[k] (4') at (2,-1){};
    \node[k] (3') at (0,-2){};
    \node[k] (0') at (2,2){};

    \draw [anno] (4') to (4);
    \draw [anno] (3') to (3);
    \draw [anno] (0') to (0);
    \end{scope}

    \begin{scope}[shift = {(8,-8)}]
    \draw [thick,<->,>=stealth'] (-3.5,0) to (3.5,0);
        \draw [thick,<->,>=stealth'] (0,-3.5) to (0,3.5);
        \foreach \x in {-3,...,3}
            {\draw [dotted] (\x,-3.5) to (\x,3.5);
            \draw [dotted] (-3.5,\x) to (3.5,\x);}
    \tikzstyle{o}=[circle,draw,fill=black, text = white,inner sep=1pt,minimum size = 1.8mm]
    \node[o] (0) at (2,3) {\Large{$0$}};
    \node[o] (1) at (-3,2) {\Large{$1$}};
    \node[o] (2) at (-1,0) {\Large{$2$}};
    \node[o] (3) at (-1,-2) {\Large{$3$}};
    \node[o] (4) at (1,-1) {\Large{$4$}};

    \node[k] (3') at (0,-1){};
    \node[k] (2') at (-1,-1){};
    \node[k] (4') at (1,-2){};

    \draw [anno] (3') to (3);
    \draw [anno] (2') to (2);
    \draw [anno] (4') to (4);
    \end{scope}

    \begin{scope}[shift = {(16,-8)}]
    \draw [thick,<->,>=stealth'] (-3.5,0) to (3.5,0);
        \draw [thick,<->,>=stealth'] (0,-3.5) to (0,3.5);
        \foreach \x in {-3,...,3}
            {\draw [dotted] (\x,-3.5) to (\x,3.5);
            \draw [dotted] (-3.5,\x) to (3.5,\x);}
    \tikzstyle{o}=[circle,draw,fill=black, text = white,inner sep=1pt,minimum size = 1.8mm]
    \node[o] (0) at (2,3) {\Large{$0$}};
    \node[o] (1) at (-2,2) {\Large{$1$}};
    \node[o] (2) at (-2,1) {\Large{$2$}};
    \node[o] (3) at (0,-2) {\Large{$3$}};
    \node[o] (4) at (1,-1) {\Large{$4$}};

    \node[k] (2') at (-1,0){};
    \node[k] (1') at (-3,2){};
    \node[k] (3') at (-1,-2){};

    \draw [anno] (3') to (3);
    \draw [anno] (2') to (2);
    \draw [anno] (1') to (1);
    \end{scope}

    \end{tikzpicture}
    \caption{This figure shows one way to visualize chip-firing on $R_{10}$. Chip-configurations correspond to 5 labeled nodes on lattice points on the complex plane. For example, the initial configuration is $(1+i, -2 + 3i, -2, 1-2i, 2)$ and the final configuration is $(2+3i, -2 + 2i, -2 + i,-2i, 1-i)$. This final configuration is reached after an A firing of 0, a (-B) firing of 2, a (-A) firing of 4, a (-A) firing of 3, and a B firing of 2. Note that there are no restrictions in this construction about multiple nodes occupying the same position; we only avoided this for the example to make things easier to draw.}
    \label{fig:geometric}
\end{center}
\end{figure}
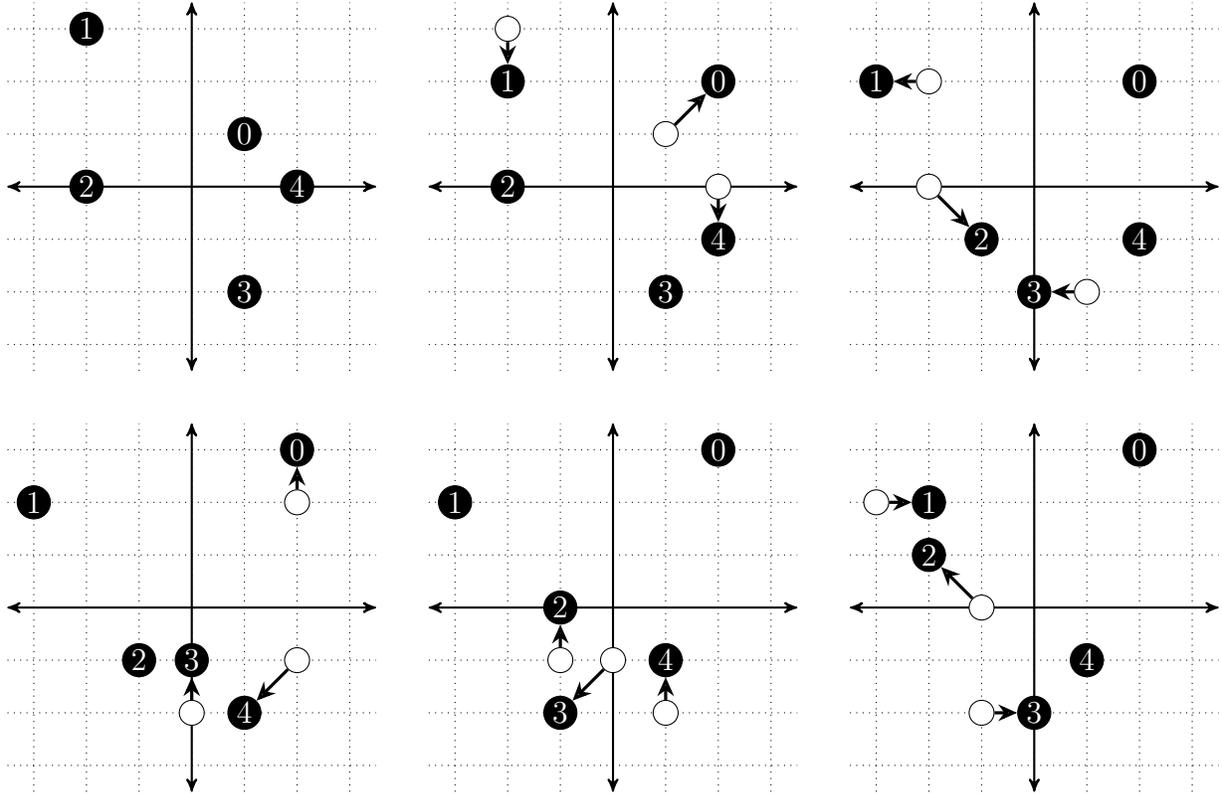

\subsection{Representatives of $S(R_{10})$}
Here, we will find a nice representative for each equivalence class of $S(R_{10})$. Note that we are using the word \emph{nice} colloquially, and we do not ascribe a precise mathematical meaning. 

\begin{lemma}\label{lem:no_i}
    Every equivalence class of $S(R_{10})$ contains chip configurations which do not have any imaginary chips. 
\end{lemma}
\begin{proof}
    Suppose that a node has $a$ real chips and $b$ imaginary chips. If $b$ is positive, we perform a (-B) firing to the node $b$ times. If $b$ is negative, we perform a B firing to the node $-b$ times. After this, the node will have no imaginary chips, and these firings do not impact the number of imaginary chips on any other node. After going through the process for all five of the nodes, there will be no imaginary chips remaining. 
\end{proof}

Now that we have shown that we can get rid of imaginary chips, we can find chip configurations with specified restrictions on the number of real chips at each node. One useful fact that will assist in this technique is the following. 

\begin{lemma}
    There exists a unique element of $S(R_{10})$ with order 2. 
\end{lemma}
\begin{proof}
    This follows immediately from the fact that $S(R_{10}) \cong (\Z/3\Z)^4 \oplus (\Z/2\Z)$. 
\end{proof}

For a chip configuration $C$, we will write $C \sim H$ if the equivalence class containing $C$ is the element of order $2$ in $S(R_{10})$. For chip configurations $C$ and $C'$, we will also write $C \sim C'$ if $C$ and $C'$ are firing equivalent. 

\begin{lemma}\label{lem:half_vals} 
    Let $C$ be the chip configuration which places 1 chip on each node. Let $C'$ be the chip configuration that places $3$ chips on one node and $0$ chips on the other nodes. Then, $C \sim H$ and $C' \sim H$. 
\end{lemma}
\begin{proof}
    One can verify that 
    \begin{equation}\label{eq:half_proof}\overline{\mathcal K}^{-1}(C) = \frac 16\begin{bmatrix}
    3 - i & 1 + i & -1 + i & -1 + i & 1 + i \\
    1 + i & 3 - i & 1 + i & -1 + i & -1 + i \\
    -1 + i & 1 + i & 3 - i & 1 + i & -1 + i \\
    -1 + i & -1 + i & 1 + i & 3 - i & 1 + i \\
    1 + i & -1 + i & -1 + i & 1 + i & 3 - i \\
\end{bmatrix} \begin{bmatrix} 1 \\ 1 \\ 1 \\ 1 \\ 1\end{bmatrix}= \frac12\begin{bmatrix}1 + i\\ 1 + i\\ 1 + i\\ 1 + i\\ 1 + i\end{bmatrix}\end{equation}
\text{ and }
\begin{equation}\label{eq:half_proof2}\overline{\mathcal K}^{-1}(C') = \frac 16\begin{bmatrix}
    3 - i & 1 + i & -1 + i & -1 + i & 1 + i \\
    1 + i & 3 - i & 1 + i & -1 + i & -1 + i \\
    -1 + i & 1 + i & 3 - i & 1 + i & -1 + i \\
    -1 + i & -1 + i & 1 + i & 3 - i & 1 + i \\
    1 + i & -1 + i & -1 + i & 1 + i & 3 - i \\
\end{bmatrix} \begin{bmatrix} 3 \\ 0 \\ 0 \\ 0 \\ 0\end{bmatrix}=\frac12\begin{bmatrix}3 -i\\ 1 + i\\ -1 + i\\ -1 + i\\ 1 + i\end{bmatrix}.\end{equation} 

In both cases, we obtain a vector that is half of an element of $\mathbb Z[i]^5$.  
\end{proof}

\begin{remark}\label{rem:specific}
    We can also use~\eqref{eq:half_proof} and~\eqref{eq:half_proof2} to determine which firings are required to add 2 chips to every node, or to add 6 chips to a specific node. 

    For the first case, we do an A firing and a B firing at every node. For the second case, we do three A  firings and a (-B) firing on the specified node, one A firing and one B firing on each adjacent node, and one (-A) firing and one B firing on each non-adjacent node. 
\end{remark}

\begin{thm}\label{thm:representatives}
    Consider the five node model for chip-firing on $R_{10}$ and choose one particular node to act as \emph{the distinguished node}. We obtain one representative for each equivalence class of $S(R_{10})$ by placing either 0 or 3 chips on the distinguished node, and either 0, 1, or 2 chips on every other node. 
\end{thm}
\begin{proof}
    To prove the theorem, we will show that every chip configuration is firing equivalent to one of the given form. Since there are 162 total chip configurations of this form, and this is also the size of $S(R_{10})$, we can then conclude that each representative must be in a different firing equivalence class.
    
    Begin with an arbitrary chip-configuration $C$. By Lemma~\ref{lem:no_i}, there exists some $C'$ with no imaginary chips such that $C' \sim C$. Suppose that $C'$ has $a$ chips on the distinguished node. We can subtract one chip from every node $a$ times to get a new configuration $C''$. By Lemma~\ref{lem:half_vals}, we know that $C'' \sim (C' + aH)$. 

    Next, for any node that doesn't have $0$, $1$, or $2$ chips, we can add or subtract 3 chips until it does. This gives a new configuration $C'''$ which has no chips on the distinguished node, and $0$, $1$, or $2$ chips on each other node. By Lemma~\ref{lem:half_vals}, if $b$ is the total number of times that we add/subtract $3$ chips, then \[C''' \sim (C'' + bH) \sim (C' + (a+b)H) \sim (C + (a+b)H).\]

    Finally, since $2H$ is equivalent to the identity, we must either have $C''' \sim C$ or $C''' \sim (C + H)$, depending on the parity of $a+b$. In the first case, we are done and $C'''$ is the desired representative. In the second case, we obtain the desired representative from $C'''$ after adding $3$ chips to the distinguished node. 
\end{proof}

The last step in the proof of Theorem~\ref{thm:representatives} required determining whether $C''' \sim C$ or $C''' \sim (C + H)$. One approach would be to check if $\overline{\mathcal K}^{-1}(C''' - C) \in \mathbb Z[i]^5$, but there is an even simpler way using the following lemma. 

\begin{lemma}\label{lem:parity}
    Let $C$ and $C'$ be chip configurations such that $C \sim (C' + H)$. Then, the total number of chips on $C$ and the total number of chips on $C'$ must have opposite parity (counting both real and imaginary chips). 
\end{lemma}
\begin{proof}
    Each of the four kinds of firings has a net effect of adding or subtracting exactly 2 chips to the configuration. In particular, a sequence of firings cannot change the parity of the total number of chips. Furthermore, each chip configuration identified in Lemma~\ref{lem:half_vals} has an odd number of total chips. The result follows. 
\end{proof}

The following algorithm gives a simple procedure to determine a representative from Theorem~\ref{thm:representatives} for each equivalence class of $S(R_{10})$. For the most part, the algorithm follows the proof of Theorem~\ref{thm:representatives}, but we avoid firing language to make the argument easier to parse, and we take advantage of Lemma~\ref{lem:parity}. 

\begin{alg}\label{alg:stabilize}
    Input: An element of $\Z[i]^5$, thought of as a chip configuration.
    \begin{enumerate}[1.]
        \item Label the nodes $v_0,v_1,v_2,v_3,v_4$ and let $\psi(v_k)$ be the number of imaginary chips at node $v_k$. 
        \item At each node $v_k$, add $\psi(v_{k}) - \psi(v_{k+1}) - \psi(v_{k-1})$ real chips (where indices are mod 5). 
        \item Remove all imaginary chips. 
        \item Introduce a boolean variable isEven.
        \item Count the total chips. If this total is even, set isEven = True. Otherwise, set isEven = False. 
        \item Subtract the number of chips at the distinguished node from every node
        \item At each node, set the number of chips to the remainder after dividing by 3. 
        \item Count the total chips. If the total is even and isEven = False, or if the total is odd and isEven = True, then add 3 chips to the distinguished node. 
    \end{enumerate}    
    Output: the resulting configuration.
\end{alg}
\begin{proof}
The first 3 steps use $B$ firings and $(-B)$ firings to get rid of all imaginary chips (see Lemma~\ref{lem:no_i}). Steps 4 and 5 keep track of the parity of the configuration so that we can use Lemma~\ref{lem:parity} to distinguish between chip configurations that differ by $H$. Steps 6 and 7 use Lemma~\ref{lem:half_vals} to add chip configurations equal to $H$. Finally, Step 8 applies our earlier parity reasoning. 
\end{proof}

\begin{example}
    Let us apply Algorithm~\ref{alg:stabilize} to find a nice representative for the equivalence class of $S(R_{10})$ that contains the chip configuration $(3+i, 4-6i, 7+i, -8-8i, 3)$ (with $v_0$ as our distinguished node). 

    In the initial configuration, we have $\psi(v_0) = 1$, $\psi(v_1) = -6$, $\psi(v_2) = 1$, and $\psi(v_3) = -8$, $\psi(v_4) = 0$. In Step 2, we add $7$ chips to $v_0$, $-8$ chips to $v_1$, $15$ chips to $v_2$, $-9$ chips to $v_3$, and $7$ chips to $v_4$. After Step 3, we have the chip configuration $(10,-4,22,-17,10)$. In Steps 4 and 5, we set isEven = False. After Steps 6 and 7, we have the chip configuration $(0, 1, 0, 0, 0)$. Finally, since the total number of chips is still odd, we do not change anything in Step 8. Thus, $(0, 1, 0, 0, 0)$ is our representative configuration. 

    We can verify that this configuration is in the same equivalence class as our original configuration by computing 
    \[ \overline{\mathcal K}^{-1}\left (\begin{bmatrix} 0 \\ 1 \\ 0 \\ 0 \\ 0\end{bmatrix}-\begin{bmatrix} 3+i \\ 4 - 6i \\ 7+i \\ -8-8i \\ 3\end{bmatrix}\right ) = \begin{bmatrix} -5 - i \\ -4 + i \\ -4 + 3i \\ 4 - i\\ -1\end{bmatrix}\]
    The resulting vector is in $\Z[i]^5$, so the two chip configurations must be firing equivalent. More explicitly, the real part of the vector tells the number of $A$ firings required at each node, while the imaginary part tells the number of $B$ firings required at each node, in order to get from the initial configuration to $(0,1,0,0,0)$.
\end{example}

\subsection{Future directions}

In this article, we defined the sandpile group of a regular matroid as a quotient group of $\Z^E$. However, there is another approach which can also be used to produce different sets of representatives of $S(R_{10})$. Instead of assigning an integer to each element of the ground set, one can instead work with \emph{orientations} on the ground set. This perspective is equivalent to requiring each node of the pentagon to have either 0 or 1 real chips and either 0 or 1 imaginary chips. 

One difference with this alternate perspective is that any ``firing move'' must maintain the restriction that every node is assigned either 0 or 1 real chips and either 0 or 1 imaginary chips. However, this is not always possible when the only firing moves allowed are $A$, $B$, $(-A)$ and $(-B)$ firings. For example, since every move has both positive and negative coefficients, no firing move would be possible from the all zeros configuration. 

Nevertheless, this problem goes away if we also allow for simultaneous firings (a fact that we will not prove here, but which follows from results in~\cite{BBY}). For example, if we perform all 5 $B$ firings simultaneously, this adds 1+i chips to each node. applying this combined move to the all zeros configuration would maintain the restriction that each node has $0$ or $1$ real chips and $0$ or $1$ complex chips.

\begin{remark}The allowable moves in the \emph{orientation} context above correspond to reversing \emph{signed circuits} or \emph{signed cocircuits} of the matroid. For many more details on this perspective, see~\cite{Gioan}. \end{remark}

\begin{open}
    Theorem~\ref{thm:representatives} describes one set of representatives for the equivalence classes of $S(R_{10})$, but these allow for more than one chip on a node. What is a ``nice'' set of representatives where each node can have at most one real chip and at most one complex chip? 
\end{open}

As discussed in Remark~\ref{rem:mtt}, one interesting property of $S(R_{10})$ is that the size of the group is the same as the number of bases of $R_{10}$. One project for future research would be to understand this relationship more concretely. For example, one could study \emph{simply transitive actions} of $S(R_{10})$ on the bases of $R_{10}$. Such actions are sometimes referred to as \emph{sandpile torsors}. 

There are several such actions known in the literature, which all rely on certain auxiliary information (such as an ordering of the ground set, or a choice of appropriate \emph{signature} for the signed circuits and cocircuits)~\cite{GLV,BBY,BSY,Ding}. 

\begin{open}
    What are some particularly nice simply transitive actions of $S(R_{10})$ on the bases of $R_{10}$? Are there properties of this particular matroid which allow for alternate descriptions of known actions (such as those cited above)? 
\end{open}

In this article, we focused very narrowly on $R_{10}$. As mentioned in the introduction, our motivation for choosing $R_{10}$ in particular is because of a celebrated result of Seymour which showed that every regular matroid can be formed from \emph{1-sums}, \emph{2-sums}, and \emph{3-sums} of graphic matroids, co-graphic matroids, and copies of $R_{10}$~\cite{Seymour}.

Let $M$ and $M'$ be two regular matroids which are represented by matrices $A$ and $A'$ respectively. The \emph{1-sum} of $M$ and $M'$, also called the \emph{direct sum}, is the matroid represented by $\begin{bmatrix}A &  A'\end{bmatrix}$. This matroid is written $M \oplus M'$. It is trivial to show that $S(M \oplus M')\cong S(M) \oplus S(M')$. 

The definition of \emph{2-sum} and \emph{3-sum} are outside the scope of this article, but can be found in~\cite{Oxley}. 

\begin{open}
    Can the sandpile group of the 2-sum of a pair of regular matroids be described in terms of the sandpile groups of the original matroids? What about when 2-sum is replaced with 3-sum? 
\end{open}
\section*{Acknowledgments}
The second author was partially supported by NSF grant DMS-2039316.

\bibliographystyle{alpha}
\bibliography{biblio}
\end{document}